\documentclass[12pt,a4paper]{amsart}
\usepackage{mathrsfs}
\usepackage{amsfonts}
\usepackage{txfonts}

\usepackage{hyperref}
\usepackage{latexsym}
\usepackage{amssymb}

\newtheorem{theorem}{Theorem}[section]
\newtheorem{lemma}[theorem]{Lemma}
\newtheorem{corollary}[theorem]{Corollary}
\newtheorem{proposition}[theorem]{Proposition}

\newtheorem{problem}[theorem]{Problem}

\theoremstyle{definition}
\newtheorem{definition}[theorem]{Definition}

\newtheorem{remark}[theorem]{Remark}

\newcommand{\R}{\mathbb{R}}
\newcommand{\N}{\mathbb{N}}
\newcommand{\eps}{\varepsilon}

\numberwithin{equation}{section}

\begin{document}

\title[countable determination of the Kuratowski measure ]
{{\bf On countable determination of the Kuratowski measure of noncompactness}}

\author{Xiaoling Chen, Lixin Cheng$^\sharp$ }

%\address{Ehmet Ablet: School of Mathematical Sciences, Xiamen University,
% Xiamen, 361005, China}
%\email{619948515@qq.com\;\;(Ehmet Ablet)}
\address{Xiaoling Chen: School of Mathematical Sciences, Xiamen University,
 Xiamen, 361005, China}
\email{30128299@qq.com\;\;(X. Chen)}
\address{ Lixin Cheng:  School of Mathematical Sciences, Xiamen University,
 Xiamen, 361005, China}
 \email{lxcheng@xmu.edu.cn\;\;(L. Cheng)}
%\address{ Quanqing Fang:  School of Mathematical Sciences, Xiamen University,
% Xiamen, 361005, China}
 % \email{815045306@qq.com\;\;(Q. Fang)}
%\address{ Zheming Zheng:  School of Mathematical Sciences, Xiamen University,
% Xiamen, 361005, China}
% \email{3184532127@qq.com\;\;(Z. Zheng)}

 \thanks{$^\sharp$ The corresponding author; support
by NSFC, grant 11731010}

%\date{}

\begin{abstract}  A long-standing question in the theory of measures of noncompactness is that for the Kuratowski measure of noncompactness $\alpha$ defined on a metric space $M$, and for every bounded subset $B\subset M$, is there a countable\;subset $B_0\subset B$ such that $\alpha(B_0)=\alpha(B)$?
In this paper, we give an affirmative answer to the question above. It is done by showing that for each nonempty set $B$ of a Banach space,
there is a countable subset $B_0\subset B$ so that $B$ is strongly finitely representable in $B_0$, and that there is a free ultrafilter $\mathcal U$ so that $B$ is affinely isometric to
a subset of the ultrapower $[{\rm co}(B_0)]_\mathcal U$ of ${\rm co}(B_0)$.
%Based on the recent development of the theory of measures of  noncompactness, we first show a representation theorem of measures of noncompactness defined on Banach spaces. Then, in terms of measure of noncompactness, %we use this theorem to give measure of noncompactness a localized setting. As their applications,
%we present  some new characterizations of the existence of global attractors for norm-to-weak continuous semigroups, and a fixed point theorem for norm-to-weak continuous mappings.
\end{abstract}

\keywords{Kuratowski's measure of noncompactness; countable determination; strongly finite representability; Banach space; metric space}

\subjclass[2010]{Primary  47H08, 46B07, 46M07, 46B04, 46B26}

\maketitle

\section{Introduction}
The study of measures of noncompactness  has continued for over 80 years. It has been shown that the theory of measures of noncompactness  was  used  in  a wide variety of topics in nonlinear analysis. %(see, for instance , Akhmerov et al.[8],  Appell[9],  Ayerbe Toledano et al.[10],  Bana$\acute{s}$ and Goebel [4,11], Djebali et al.[12]  and Meskhi [13]).
Roughly speaking, a measure of noncompactness $\mu$ is a   nonnegative function defined on the family $\mathscr B(M)$ consisting of all bounded subset of a  metric space, in particular, a Banach space $M$ and satisfies some specific properties such as sublinearity,  non-decreasing monotonicity in the order of  the set inclusion, and the (most important) noncompactness that $\mu(B)=0$ if and only if $B$ is relatively compact in $M$.

The first measure of noncompactness $\alpha$ was introduced and studied by K. Kuratowski \cite{ku} in 1930, which is now called the Kuratowski measure of noncompactness:
\begin{equation}\alpha(B)=\inf\{\ d>0: B\subset\cup_{j\in F}E_j, \;F^\sharp<\infty,\;d(E_j)\leq d\ \},\;B\in\mathscr B(M),\end{equation}
where $F^\sharp$ is the cardinality of $F\subset\N$, and $d(E_j)$ denotes  the diameter of $E_j\subset M$. This measure has been widely used.
The earliest successful application of the Kuratowski measure  was applied in the fixed point theory.
%One of the most important theorems in fixed point theory is the famous Brouwer's theorem, which shows that  every continuous mapping from the unit ball of $\mathbb{R}^{n}$ into itself has a fixed point.
%A generalization of Brouwer's theorem is Schauder's theorem, which shows that every  compact mapping from a bounded, closed and convex subset $C$ of a Banach space $X$ into $C$ has a fixed point.
%As we have seen in Schauder's  theorem,  compactness plays an essential role.
In 1955, G. Darbo \cite{da}  extended  the Schauder fixed point theorem to noncompact operators  named set-contractive operators. %that is ,operators which satisfy $\alpha (T(A))\leq k \alpha (A)$ with $k <1$, $\alpha(\cdot)$ being the Kuratowski measure of noncompactness.
%After this result several fixed point theorems for set-contractive or condensing operators were proved (see for exmple B.N.Sadovskii[7]).
Since then, the study of measures of noncompactness and of their applications has become an active research area, and various measures of noncompactness have appeared. Among many other measures, the Hausdorff measure of noncompactness (or, the ball measure of noncompactness) $\beta$ is another commonly used measure (introduced by I. Gohberg, L.S. Go{\l}den\'{s}shte\u{\i}n and A.S.Markus \cite{go} in 1957). It is defined for $B\in\mathscr B(M)$ by
\begin{equation}\beta(B)=\inf\{r>0: B\subset\cup_{x\in F}B(x,r)\;{\rm for\;some\;finite\;}F\subset M\},\end{equation}
where $B(x,r)$ denotes the closed ball centered at $x$ with radius $r$.

It is easy to observe that if $\mu$ is either the Kuratowski measure $\alpha$ or the Hausdorff measure $\beta$, then it satisfies the following three conditions.

(1) $B\in\mathscr B(M),\;\mu(B)=0 \Longleftrightarrow B$ is relatively compact;

(2) $A, B\in\mathscr B(M)\; {\rm with\;}A\supset B\;\Longrightarrow \mu(A)\geq\mu(B);$

(3) $A, B\in\mathscr B(M) \Longrightarrow\;\mu(A\cup B)=\mu(A)\vee\mu(B),$ and

%(4) $B\in\mathscr B(M) \Longrightarrow \mu(\overline{B})= \mu(B)$, where $\overline{B}$ denotes the closure of $B$.

\noindent
If, in addition, $M$ is a Banach space,  then

(4) $B\in\mathscr B(M)\Longrightarrow \mu({\rm co}{(B)})= \mu(B)$;

(5) $B\in\mathscr B(M)\Longrightarrow \mu((kB))= |k|\mu(B)$, \;$\forall\;{\rm scalar}\; k;$

(6) $A, B\in\mathscr B(M) \Longrightarrow\;\mu(A+ B)\leq\mu(A)+\mu(B).$

A measure $\mu$ of noncompactness defined on a Banach space $M$ satisfying the six properties above is called a regular measure (see, \cite{ba, ba2}).
It is easy to observe that both the Kuratowski measure $\alpha$ and the Hausdorff measure $\beta$ are regular measures. \\

In the theory of measures of noncompactness, the following  ``countable determination question" is a long-standing problem (see, for example, \cite[Subsection 1.4.3, p.19]{ak}).
 \begin{problem} Given a measure of noncompactness $\mu$ on a complete metric space $M$, is the following statement true?
\begin{equation}
\forall B\in\mathscr B(M),\;\exists \;{\rm countable\;subset\;} B_0\subset B\;{\rm so\;that\;}\mu(B_0)=\mu(B).
\end{equation}
\end{problem}
This question is not only natural but also important. For example, if the metric space $M$ in question is a Banach space consisting of measurable functions (say, an $L_p$-space),  and if $B=\{f_n\}\in\mathscr B(M)$ is a countable subset of $M$, then all the functions
$\sup B\equiv\sup_nf_n$, $\inf B\equiv\inf_nf_n$, $\limsup_nf_n$ and $\liminf_nf_n$ are measurable. Besides, it is usually easier to deal with a sequence than to handle an uncountable subset. See, for example, \cite{ag, bo}.
B.N. Sadovskii (\cite{sa, sa2}) first studied the countable determination  question, and gave a negative answer to it for the Hausdorff measure $\beta$. Indeed,  he defined a ``sequential measure" $\tilde{\beta}$ with respect to $\beta$ by
\begin{equation}
\tilde{\beta}(B)=\sup\lbrace\beta(C):C\;{\rm is\;a\; countable\; subset\; of\;} B \rbrace,\;B\in\mathscr B(M),
\end{equation}
 and showed  the following (sharp) inequalities.
 \begin{equation} \frac{1}{2}\beta(B)\leq\tilde{\beta}(B)\leq\beta(B),\; B\in\mathscr B(M).\end{equation}
But he constructed a Banach space satisfying that there is a bounded subset $B$ in it such that $\tilde{\beta}(B)=\frac{1}2<1=\beta(B)$ (see, also, \cite[\S.1.4]{ak}).
However, the following question remains open.
\begin{problem}
For the Kuratowski measure of noncompactness $\alpha$ on a metric $M$, is the following assertion true?
\begin{equation}
\forall B\in\mathscr B(M),\;\exists \;{\rm countable\;subset\;} B_0\subset B\;{\rm so\;that\;}\alpha(B_0)=\alpha(B).
\end{equation}
\end{problem}
In this paper, we will give this question an affirmative answer. \\

This paper is organized as follows. In Section 2, we first recall some definitions related to a localized setting of ``finite representability". Then we introduce a new notion, namely, strongly finite representability by substituting ``polyhedrons" for ``simplexes" in the notion of ``finite representability", and through a lemma,  we show that  the strongly finite representability implies the finite representability, and it is properly stronger than the finite representability by a simple example. In the third section,  we first recall some notions and basic properties related to ultraproducts of sets in Banach spaces. Then we conclude this section
by a lemma (Lemma 3.3) stating that if a subset $A$ of a Banach space is strongly finitely representable in a convex subset $B$ of another Banach space, then there is a free ultrafilter $\mathcal U$ such that $A$ is affinely isometric to a subset of the ultrapower $(B)_\mathcal U$ of $B$. In Section 4, after showing eight  lemmas, we further prove the main result of this section: For every subset $B$ of a Banach space, there is a countable subset $B_0$ of $B$ such that $B$ is strongly finitely representable in $B_0$. In Section 5 (the last section), we finally show  the result stated in the abstract of the paper: For every (bounded) subset $B$ of a metric space, there exists a countable subset $B_0\subset B$ such that the Kuratowski measure of $B_0$ coincides with the Kuratowski measure of $B$.

\section{Localized finite representability in Banach spaces }
In this section, we will assume that $X$ is a Banach space, and $X^*$ its dual. $B_X$ denotes the closed unit ball of $X$. For a subset $A\subset X$, ${\rm co}(A)$ (resp. $\overline{A}$,  ${\rm aff}(A)$) stands for the convex hull (resp. the closure, the affine hull) of $A$.

Recall that a Banach space $X$ is said to be finitely representable in another Banach space $Y$ provided that for all $\eps>0$ and for every finite dimensional subspace $F\subset X$ there exist a (finite dimensional) subspace $G\subset Y$ and a linear isomorphism $T: F\rightarrow G$ so that $\|T\|\cdot\|T^{-1}\|<1+\eps$. The following notion is a generalization of the classical finite representability of Banach spaces to general subsets, which was introduced by Cheng et al. \cite{cheng2} (for convex subsets) and \cite{cheng1} (for general subsets). It is done by substituting ``simplexes" for the ``finite dimensional subspaces".  An $n$-simplex  in a linear space $X$ is a convex set $S$ satisfying that there exist $n+1$ affinely independent vectors $x_j\in X,\;j=0,1,\cdots,n$, i.e. $x_1-x_0, x_2-x_0,\cdots,x_n-x_0$ are linearly independent, such that $S={\rm co}(x_0,x_1,\cdots,x_n)$, the convex hull of $(x_j)_{j=0}^n$.
\begin{definition}
Suppose that $X$ and $Y$ are Banach spaces, and $A\subset X$, $B\subset Y$ are two subsets.

i) Given $\eps>0$, the set  $A$ is said to be $\eps$-finitely representable in the set $B$, if for all  $n\in\N$ and for every $n$-simplex $S_n={\rm co}(x_0,x_1,\cdots,x_n)$ with the vertices $(x_j)_{j=0}^n\subset A$, there exist an $n$-simplex $S_n^\prime={\rm co}(y_0,y_1,\cdots,y_n)$ with vertices $(y_j)_{j=0}^n\subset B$ and an affine isomorphism $T: {\rm aff}(S_n)\rightarrow {\rm aff}(S_n^\prime) $ satisfying that $T(x_j)_{j=0}^n=(y_j)_{j=0}^n$ (or, equivalent, $TS_n=S^\prime_n$) and that
\begin{equation}
(1-\eps)\|x-y\|\leq\|Tx-Ty\|\leq(1+\eps)\|x-y\|,\;\forall\;x,y\in {\rm aff}(S_n),
\end{equation}
where $\|T\|=\sup\{\frac{\|Tx-Ty\|}{\|x-y\|}: x\neq y\in{\rm aff}(S_n)\}.$

ii) If, in addition,  for all $\eps>0$, $A$ is  $\eps$-finitely representable in $B$, then we say that $A$ is finitely representable in  $B$.

iii) In particular, if $A$ itself is finite, and if $A$ is finitely representable in  $B$, then we simply say that $A$ is  representable in  $B$.
\end{definition}

It is easy to observe that $X$ is finitely representable in $Y$ if and only if the unit ball $B_X$ of $X$ is finitely representable in the unit ball $B_Y$ of $Y$.\\

For a Banach space $X$, we use $\mathscr B(X)$ to denote the collection of all nonempty bounded subsets of $X$ endowed with the Hausdorff metric $d_\mathfrak{H}$:
\begin{equation} d_\mathfrak{H}(A,B)\equiv\inf\{r>0: A\subset B+rB_X, B\subset A+rB_X\},\;\;A, B\in\mathscr B(X),\end{equation}
where $B_X$ denotes the closed unit ball of $X$.

\begin{definition} Let $H$ be a convex set in a Banach space $X$. Then

  i) $H$ is said to be a convex polyhedron provided there exist $z_1,z_2,\cdots, z_k\in X$ for some $k\in\N$ such that $H={\rm co}(z_1,z_2,\cdots,z_k);$ In this case, the set ${\rm extr}(H) (\subset(z_1,z_2,\cdots,z_m))$ of all extreme points of $H$ is called the vertex set of $H$, and denoted by ${\rm vert}(H) (={\rm extr}(H))$.

 ii)  A convex polyhedron $H$ is called an $n$-dimensional  convex polyhedron if its affine hull ${\rm aff}(H)$ is an $n$-dimensional affine subspace of $X$, i.e. ${\rm dim}[{\rm aff}(H)]=n$.

 iii)  We say that an $n$-dimensional  convex polyhedron $H$ is an $(n,m)$-polyhedron for some non-negative integer $m$  if $${\rm vert}(H)=(x_0, x_1, \cdots, x_n, x_{n+1},\cdots, x_{n+m}).$$
Since the vertex set ${\rm vert}(H)$ of the $(n,m)$-polyhedron $H$  always contains a maximal affinely independent subset of $n+1$ elements, we write an $(n,m)$-polyhedron $H$ as
 \begin{equation}
 H={\rm co}(x_0,x_1,\cdots,x_n,x_{n+1},\cdots,x_{n+m}),
 \end{equation}
where $(x_0, x_1, \cdots, x_n)$ is a maximal affinely independent subset of ${\rm vert}(H)$.  In this case, we have
\begin{equation}
{\rm aff}(H)={\rm aff}(x_0, x_1, \cdots, x_n).
\end{equation}
\end{definition}
\begin{remark}%For distinction, we often write an $(n,m)$-convex polyhedron $H$ as
 %\begin{equation}
% H={\rm co}(x_0,x_1,\cdots,x_n,u_{1},\cdots,u_{m}),
 %\end{equation}
 %instead of (2.3).

Unless stated otherwise, by an $(n,m)$-polyhedron $$H=H(x_0,x_1,\cdots,x_n,x_{n+1},\cdots,x_{n+m}),$$ we always mean that $${\rm vert}(H)=(x_0,x_1,\cdots,x_n,x_{n+1},\cdots,x_{n+m})$$ with $H={\rm co}(x_0,x_1,\cdots,x_{n+m})$, that  $S\equiv{\rm co}(x_0,x_1,\cdots,x_n)$ is an $n$-simplex, and that ${\rm aff}(S)={\rm aff}(H).$
 \end{remark}
 \begin{lemma}
 Suppose that $H_1=H(x_0,x_1,\cdots,x_n,x_{n+1},\cdots,x_{n+m})$ and $H_2=H(y_0,y_1,\cdots,y_n,y_{n+1},\cdots,y_{n+m})$ are two $(n,m)$-polyhedrons in a Banach space $X$.
 Let $$\delta_1=\min\{\|x_i-x_j\|:0\leq i\neq j\leq n+m\}>0,$$  $$\delta_2=\min\{\|y_i-y_j\|:0\leq i\neq j\leq n+m\}>0,$$ and let \begin{equation}0<2\eps<\delta=\min\{\delta_1,\delta_2\}.\end{equation}
 Then
 \begin{equation}d_\mathfrak{H}({\rm vert}(H_1),{\rm vert}(H_2))<\eps\end{equation}
 if and only if there is a permutation $\pi:\{0,1,\cdots,n+m\}\rightarrow \{0,1,\cdots,n+m\}$ such that
 \begin{equation}\|x_{j}-y_{\pi(j)}\|<\eps,\;j=0,1,\cdots,n+m.\end{equation}
 \end{lemma}
 \begin{proof}
 Sufficiency. Assume that there is a permutation $$\pi:\{0,1,\cdots,n+m\}\rightarrow \{0,1,\cdots,n+m\}$$ such that (2.7) holds.
 %\begin{equation}\|x_{j}-y_{\pi(j)}\|<\eps,\;j=0,1,\cdots,n+m.\end{equation}
 Then we obtain \begin{equation}{\rm vert}(H_1)\subset {\rm vert}(H_2)+\eps B_X,\; {\rm and\;}\; {\rm vert}(H_2)\subset {\rm vert}(H_1)+\eps B_X.\end{equation}
 %These imply that
 %\begin{equation}H_1\subset H_2+\eps B_X,\; {\rm and\;} H_2\subset H_1+\eps B_X\end{equation}
 Thus, (2,6) holds.%$d_{\mathfrak{H}}(H_1,H_2)<\eps$.

 Necessity. Assume that (2.6) is true. Since it is equivalent to (2.8), for each $x_j\in{\rm vert}(H_1)$, there is $z_j\in{\rm vert}(H_2)$ such that
 $\|x_j-z_j\|<\eps.$ We denote by $y_{\pi(j)}=z_j$. Then it remains to show that $\pi:{\rm vert}(H_1)\rightarrow {\rm vert}(H_2)$ is a bijection. Since  the cardinalities of both ${\rm vert}(H_1)$ and ${\rm vert}(H_2)$ are $n+m+1$, it suffices to prove that $\pi$ is injective.
 For any fixed $0\leq j\leq n+m$, let $z_{1j},z_{2j}\in{\rm vert}(H_2)$ be two different points such that $\|x_j-z_{1j}\|,\;\|x_j-z_{2j}\|<\eps.$ Then by (2.5) $$2\eps<\delta\leq\|z_{1j}-z_{2j}\|\leq\|x_j-z_{1j}\|+\|x_j-z_{2j}\|<2\eps,$$
 and this is a contradiction.
 \end{proof}

  In the following, we will introduce a notion, namely, strong finite representability.  It is done by substituting $(n,m)$-polyhedrons for $n$-simplexes in Definition 2.1.

\begin{definition}
Suppose that $X$ and $Y$ are Banach spaces, and $A\subset X$, $B\subset Y$ are two subsets.

i) Given $\eps>0$, the set  $A$ is said to be $\eps$-strongly finitely representable in the set $B$, if for every pair $n,m$ of non-negative integers, and for every $(n,m)$-polyhedron $H\subset X$ with ${\rm vert}(H)\subset A$, there exist an $(n,m)$-polyhedron $H^\prime\subset Y$ with ${\rm vert}(H^\prime)\subset B$ and an affine isomorphism $T: {\rm aff}(H)\rightarrow {\rm aff}(H^\prime)$ such that
\begin{equation}
(1-\eps)\|x-y\|\leq\|Tx-Ty\|\leq(1+\eps)\|x-y\|,\;\;\forall\;x,y\in{\rm aff}(H),
\end{equation}
and
\begin{equation}
d_\mathfrak{H}\Big({\rm vert}(T(H)),{\rm vert}(H^\prime)\Big)<\eps,
\end{equation}
where $d_\mathfrak{H}$ is the Hausdorff metric on $\mathscr{B}(Y)$, the set of all nonempty bounded subsets of $Y$.

ii) $A$ is said to be strongly finitely representable in  $B$, if $A$ is  $\eps$-strongly finitely representable in  $B$ for all $\eps>0$.

iii) In particular, if $A$ itself is finite, and if $A$ is strongly finitely representable in  $B$, then we simply say that $A$ is strongly representable in  $B$.
\end{definition}

\begin{lemma}
Assume that $A\subset X$ is strongly finitely representable in $B\subset Y$. Then for all $\eps>0$ and for every $(n,m)$-polyhedron $$H=H(x_0,x_1,\cdots,x_{n+m})\;\;{\rm with\;}\;{\rm vert}(H)\subset A,$$ there exist an
$(n,m)$-polyhedron $$H^\prime=H^\prime(y_0,y_1,\cdots,y_{n+m})\;\;{\rm with}\; {\rm vert}(H^\prime)\subset B$$
and an affine isomorphism $T:{\rm aff}(H)\rightarrow{\rm aff}(H^\prime)$ such that
\begin{equation}
(1-\eps)\|x-y\|\leq\|Tx-Ty\|\leq(1+\eps)\|x-y\|,\;\;\forall\;x,y\in{\rm aff}(H),
\end{equation}
\begin{equation}
T(S)=S^\prime,\;\;{\rm and\;\;} d_\mathfrak{H}\Big({\rm vert}(T(H)),{\rm vert}(H^\prime)\Big)<\eps,
\end{equation}
where $d_\mathfrak{H}$ is the Hausdorff metric on $\mathscr B(Y)$, and $S$ (resp. $S^\prime$) is the $n$-simplex ${\rm co}(x_0,x_1,\cdots,x_n)$ (resp. ${\rm co}(y_0,y_1,\cdots,y_n)$).
\end{lemma}
\begin{proof}
Let $x_c=\sum_{i=0}^n\lambda_ix_i$ (for some $\lambda_i>0$, $i=0,1,\cdots, n$ with $\sum_{i=0}^n\lambda_i=1$) be a barycenter of the $n$-simplex $S$.
%\begin{equation}S={\rm co}(x_0,x_1,\cdots,x_n).\end{equation}
Then there is $r>0$ so that\begin{equation}r(H-x_c)\subset S-x_c.\end{equation} %where $X_H={\rm aff}(H)$ and $B_{X_H}$ denotes the closed unit ball centered at $x_c$.

Since  $A\subset X$ is strongly finitely representable in $B\subset Y$, for each $k\in\N$, there exist an $(n,m)$-polyhedron $H_k\subset Y$ with $${\rm vert}(H_k)=(z_{0k},z_{1k},\cdots,z_{(n+m)k})\subset B$$ and an affine isomorphism $T_k:{\rm aff}(H)\rightarrow{\rm aff}(H_k)$ satisfying that
\begin{equation}
(1-\frac{1}k)\|x-y\|\leq\|T_k x-T_k y\|\leq(1+\frac{1}k)\|x-y\|,\;\;\forall\;x,y\in{\rm aff}(H),
\end{equation}
and that
\begin{equation}
d_\mathfrak{H}\Big({\rm vert}(T_k(H)),{\rm vert}(H_k)\Big)<\frac{1}k.
\end{equation}
By Lemma 2.4, for every sufficiently large $k\in\N$, there is a permutation $\pi: \{0,1,\cdots,n+m\}\rightarrow \{0,1,\cdots,n+m\}$ %and $\sigma: \{1,2,\cdots,m\}\rightarrow\{1,2,\cdots,m\}$
such that
$$\|T_k x_{\pi(j)}-z_{jk}\|<\frac{1}k,\;\;j=0,1,\cdots,n+m.$$ %and $$\|T_k u_{\sigma(j)}-w_{jk}\|<\frac{1}k,\;\;j=1,2,\cdots,m.$$
Without loss of generality, we can assume
\begin{equation}\|T_k x_{j}-z_{jk}\|<\frac{1}k,\;\;j=0,1,\cdots,n+m;\;k\in\N.\end{equation} %and \begin{equation}\|T_k u_{j}-w_{jk}\|<\frac{1}k,\;\;j=1,2,\cdots,m.\end{equation}
Since for each $k\in\N$,  $T_k:{\rm aff}(H)\rightarrow{\rm aff}(H_k)$ is an affine isomorphism, and since $(x_0,x_1,\cdots,x_n)$ is a maximal affinely independent subset of ${\rm vert}(H)$, $(T_kx_0,T_kx_1,\cdots,T_kx_n)$ is again a maximal affinely independent subset of the vertex set of ${\rm vert}(T_k(H))=(T_kx_0,T_kx_1,\cdots,T_kx_{n+m})$ of the $(n,m)$-polyhedron $T_k(H)$.

It follows from (2.13), (2.14) and (2.16) that for each sufficiently large $k\in\N$, the $n+1$ vectors $z_{0k},z_{1k},\cdots,z_{nk}$ are affinely independent. Therefore, $(z_{0k},z_{1k},\cdots,z_{nk})$ is maximal affinely independent subsets of $$(z_{0k},z_{1k},\cdots,z_{nk},z_{(n+1)k},\cdots,z_{(n+m)k}).$$
We write \begin{equation}v_{jk}=T_kx_{j},\;j=0,1,\cdots,n+m;\end{equation}  %\begin{equation}w_{jk}=T_ku_{j}, j=1,2,\cdots,m;  \end{equation}
and let
 \begin{equation}S_k={\rm co}(z_{0k},z_{1k},\cdots,z_{nk}). \end{equation}
 Then it follows from (2.16) and (2.17) that
 \begin{equation}
 T_k(S)={\rm co}(v_{0k},v_{1k},\cdots,v_{nk}),
 \end{equation}
 and
 \begin{equation}d_\mathfrak{H}\Big({\rm vert}(T_k(S)),{\rm vert}(S_k)\Big)\rightarrow 0,\;{\rm as\;}k\rightarrow\infty.\end{equation}
 Since $(z_{jk})_{j=0}^n$  is also a maximal affinely independent subset of ${\rm vert}(H_k)$,
we can define an affine isomorphism $U_k: {\rm aff}[T_k(H)](={\rm aff}(H_k))\rightarrow{\rm aff}(H_k)$ satisfying
\begin{equation}U_k(v_{jk})=z_{jk},\;j=0,1,\cdots,n.\end{equation}
Therefore,
\begin{equation}W_k(S)=(U_kT_k)(S)=U_k(T_k(S))=S_k,\end{equation}
where $W_k=U_kT_k:{\rm aff}(H)\rightarrow{\rm aff}(H_k)$ is an affine isomorphism.
Thus, (2.20) is equivalent to that
\begin{equation}d_\mathfrak{H}\Big({\rm vert}(T_k(S)),{\rm vert}(W_k(S))\Big)\rightarrow 0,\;{\rm as\;}k\rightarrow\infty.\end{equation}
Consequently, by Lemma 2.4,
\begin{equation}
\|W_k-T_k\|_{{\rm aff}(H)}\rightarrow0,\;{\rm as\;}k\rightarrow\infty.
\end{equation}
This, incorporating of (2.14) imply that for all sufficiently large $k\in\N$
\begin{equation}(1-\eps)\|x-y\|\leq\|W_kx-W_ky\|\leq(1+\eps)\|x-y\|,\;\;\forall x,y\in {\rm aff}(H),\end{equation}
and which is equivalent to
\begin{equation}(1-\eps)\|x-y\|\leq\|U_kx-U_ky\|\leq(1+\eps)\|x-y\|,\;\;\forall x,y\in {\rm aff}(H_k),\end{equation}
Note that ${\rm aff}(T_k(H))={\rm aff}(H_k)$, ${\rm vert}(T_k(H))=(T_{k}x_j)_{j=0}^{n+m}$ and $${\rm vert}(W_k(H))=(z_{jk})_{j=0}^{n}\cup (W_kx_j)_{j=n+1}^{n+m}.$$ Then it follows from (2.24) that for all sufficiently large $k\in\N$,
\begin{equation}d_\mathfrak{H}\Big({\rm vert}(W_k(H)),{\rm vert}(T_k(H))\Big)<\frac{\eps}2.\end{equation}
This, incorporating of (2.15), entails that for all sufficiently large $k\geq\frac{2}\eps$,
\begin{equation}
\begin{array}{cc}
                d_\mathfrak{H}\Big({\rm vert}(W_k(H)),{\rm vert}(H_k)\Big)
                \leq d_\mathfrak{H}\Big({\rm vert}(W_k(H)),{\rm vert}(T_k(H))\Big)\\
                +d_\mathfrak{H}\Big({\rm vert}(T_k(H)),{\rm vert}(H_k)\Big)<\frac{\eps}2+\frac{1}k<\eps.
                \end{array}
\end{equation}

Now, we fix  such a sufficiently large $k\in\N$, and let $T=W_k$, $H^\prime=H_k$ and $S^\prime=S_k$.  Then we finish the proof by (2.22), (2.25) and (2.28).
\end{proof}
The following result follows directly from Lemma 2.6.
\begin{corollary}
Suppose that $A, B$ are two subsets of a Banach space $X$. If $A$ is strongly finitely representable in $B$, then it is finitely representable in $B$.
\end{corollary}
\begin{proposition}
Suppose that $X$ and $Y$ are Banach spaces. Then $X$ is strongly finitely representable in $Y$ if (and only if) $X$ is  finitely representable in $Y$.
\end{proposition}
\begin{proof}
Since $X$ is  finitely representable in $Y$, for any $n  (\leq{\rm dim}(X))$ dimensional subspace $X_n$ of $X$, and for any $\eps>0$, there exist an $n$-dimensional subspace $Y_n$ of $Y$ and a linear isomorphism $T:X_n\rightarrow Y_n$
such that the Mazur distance $1\leq\|T\|\|T^{-1}\|\leq1+\eps$. Without loss of generality, we can assume $\|T^{-1}\|=1$,
\begin{equation}\|x-y\|=\|T^{-1}(Tx-Ty)\|\leq\|Tx-Ty\|\leq(1+\eps)\|x-y\|,\;\;x,y\in X_n.\end{equation}
Given any $(n,m)$-polyhedron $H\subset X$, let $X_k={\rm span}(H)$, where $k\in\{n,n+1\}$ is the dimensional of ${\rm span}(H)$. Then  there exist an $k$-dimensional subspace $Y_k$ of $Y$ and a linear isomorphism $T:X_k\rightarrow Y_k$ satisfying\
\begin{equation}\|x-y\|=\|T^{-1}(Tx-Ty)\|\leq\|Tx-Ty\|\leq(1+\eps)\|x-y\|,\;\;x,y\in X_k.\end{equation}
 Then $H^\prime\equiv T(H)$ is again an $(n,m)$-polyhedron in $Y$. We finish the proof by this fact and (2.30).
\end{proof}
\begin{remark}
 It follows from Corollary 2.7 and Proposition 2.8 that the strong finite representability implies the finite representability, and the two notions are equivalent whenever both $A$ and $B$ are the whole spaces. But they are not equivalent in general. For example,
let $X=Y=\ell^2_2$, $A=\{(0,0),(1,0),(0,1),(1,1)\}$, $B=\{(0,0),(1,0),(0,1)\}$. Then $A$ is finitely representable in $B$. However, $A$ is not strongly finitely representable in $B$.
\end{remark}

 %Note that for every $n$-dimensional convex polyhedron $H={\rm co}(z_1,z_2,\cdots,z_m)$, we have $H={\rm co}({\rm extr}(H))$ and there exists a maximal affinely independent subset $(x_0,x_1,\cdots,x_n)$ of ${\rm extr}(H)\subset(z_1,z_2,\cdots,z_m)$ so that $H\subset{\rm aff}(x_0,x_1,\cdots,x_n)={\rm aff}(H)$. Thus, every $n$-dimensional convex polyhedron $H$ has the form $$H={\rm co}(x_0,x_1,\cdots,x_n;u_1,\cdots,u_m),$$ where $(x_0,x_1,\cdots,x_n;u_1,\cdots,u_m)={\rm extr}(H)$, and $(x_0,x_1,\cdots,x_n)$ is a maximal affinely independent subset of ${\rm extr}(H)$. Since $S\equiv{\rm co}(x_0,x_1,\cdots,x_n)$ is an $n$-simplex, we also denote $H$ by ${\rm co}(S;u_1,u_2,\cdots,u_m)$.

\section{Ultraproducts of subsets in Banach spaces}
A filter $\mathcal F$ on a set $I$ is a family of subsets of $I$ satisfying that a) $\emptyset\notin\mathcal F$; b) for all $n\in\N$, $(F_j)_{j=1}^n\subset\mathcal F\Longrightarrow \cap_jF_j\in\mathcal F;$ and c) $A\in\mathcal F,\;{\rm and\;} A\subset B\subset I\Longrightarrow B\in\mathcal F$. A filter $\mathcal F$ is called a free filter if $\bigcap\{F\in\mathcal F\}=\emptyset$. We say that a filter $\mathcal F$ is an ultrafilter if it satisfies that  either $S\in\mathcal F$ or $I\setminus S\in\mathcal F$ for any $S\subset I$.  We will always use $\mathcal U$ to denote an ultrafilter on a set $I$.

Let $X_i,\;i\in I$ be Banach spaces, $A_i\subset X_i,\;i\in I,$ and let $\prod_IA_i$ be the Cartesian product of these sets, i.e. the set of all families $(a_i)_{i\in I}$ with $a_i\in A_i$. Two families $(a_i), \;(b_i)$ are said to be equivalent with respect to the ultrafilter $\mathcal U$, if for every $\eps>0$
$$\{i\in I: \|a_i-b_i\|<\eps\}\in\mathcal U.$$
This defines an equivalence relation on $\prod_IA_i$.

\begin{definition}
The set of all equivalence classes of $\prod_IA_i$ with respect to the ultrafilter $\mathcal U$ is called the ultraproduct of the subsets $(A_i)_{i\in I}$ denoted by $(A_i)_\mathcal U.$
In particular, if $A_i\equiv A$ for all $i\in I$, we simply denote $(A_i)_\mathcal U$ by $(A)_\mathcal U,$ and call it the ultrapower of $A$.
\end{definition}
By Definition 3.1, the next property follows easily.
\begin{proposition}
Let $E_i,\; {i\in I}$ be a family of Banach spaces, and $A_i, B_i\subset E_i, \;{i\in I}$. Then for every ultrafilter $\mathcal U$ on $I$,
\begin{eqnarray}\nonumber
(A_i\cup B_i)_\mathcal U=(A_i)_\mathcal U\cup (B_i)_\mathcal U;\\\nonumber (A_i\cap B_i)_\mathcal U=(A_i)_\mathcal U\cap (B_i)_\mathcal U;\\\nonumber
(A_i\setminus B_i)_\mathcal U=(A_i)_\mathcal U\setminus (B_i)_\mathcal U.\;\;
\end{eqnarray}
\end{proposition}
%\subsection{A lemma concerning strongly finite representability and ultraproducts}
It was shown in \cite[Prop.2.4]{cheng1} that for two Banach spaces $X$ and $Y$, and for two bounded subsets $A\subset X$ and $B\subset Y$, if $A$ is finitely representable in $B$, then for every affinely independent subset $A_0$ of $A$, there exist a free ultrafilter $\mathcal U$ on some index set $I$ and an affine isometry $T: {\rm aff}(A_0)\rightarrow ({\rm aff}(B))_\mathcal U$ with $T(A_0)\subset (B)_\mathcal U.$  If $A$ is strongly finitely representable in $B$, then we can further show the following result.
\begin{lemma}
Suppose that $A$ is a  subset of a Banach space $X$, $B$ is a convex subset of a Banach space $Y$. If $A$ is strongly finitely representable in $B$, then  there exist a free ultrafilter $\mathcal U$ on some index set $I$ and an affine isometry $T: {\rm aff}(A)\rightarrow ({\rm aff}(B))_\mathcal U$ such that  $T(A)\subset (B)_\mathcal U.$

%ii) If $0\in A$, and every linearly independent subset of $A$ is finitely representable in $B$, then  there exists a free ultrafilter $\mathcal U$ on some index set $I$ and an affine isometry $T: {\rm span}(A)\rightarrow ({\rm aff}(B))_\mathcal U$ with $T(A)\subset (B)_\mathcal U.$
 \end{lemma}

\begin{proof}
Case I. $A$ is finite, say, $A=(z_1,z_2,\cdots,z_k)$ for some $k\in\N$. Let $H={\rm co}(z_1,\cdots,z_k)$. Then $H$ is an $(n,m)$-polyhedron  for some integers $n,m\geq0$. Let $$(x_0,x_1,\cdots,x_n,u_1,\cdots,u_m)={\rm vert}(H),$$ where $(x_0,x_1,\cdots,x_n)$ is  a maximal affinely independent subset of ${\rm vert}(H)$. Let $S={\rm co}(x_0,x_1,\cdots,x_n).$  Since $A$ is strongly finitely representable in $B$, by Lemma 2.6, for every $j\geq1$, there exist an $(n,m)$-polyhedron $$H^\prime_j={\rm co}(y_{0j},y_{1j},\cdots,y_{nj},v_{1j},\cdots,v_{mj})$$ with $${\rm vert}(H_j^\prime)=(y_{0j},y_{1j},\cdots,y_{nj},v_{1j},\cdots,v_{mj})\subset B,$$ (where $(y_{0j},y_{1j},\cdots,y_{nj})$ is a maximal affinely independent subset of ${\rm vert}(H_j^\prime)$) and an affine mapping $T_{H,j}: {\rm aff}(H)\rightarrow {\rm aff}(H^\prime_j)$ with $$T_{H,j}(S)={\rm co}(y_{0j},y_{1j},\cdots,y_{nj})\equiv S^\prime_j$$ such that
\begin{equation}
(1-\frac{1}j)\|x-y\|\leq\|T_{H,j}x-T_{H,j}y\|\leq(1+\frac{1}j)\|x-y\|,\;\forall x, y\in {\rm aff}(H),
\end{equation}
and the Hausdorff metric
\begin{equation}
d_\mathfrak{H}\Big({\rm vert}(T_{H,j}(H)),{\rm vert}(H^\prime_j)\Big)<\frac{1}j.
\end{equation}
Since ${\rm vert}(H_j^\prime)\subset B$, and since $B$ is convex, $H^\prime_j\subset B.$

For any fixed free ultrafilter $\mathcal N$ on $\N$, we define an affine mapping $T_H: {\rm aff}(H)\rightarrow ({\rm aff}(B))_\mathcal N$ by
\begin{equation}
T_H(x)=(T_{H,j}(x))_\mathcal N,\;\;x\in{\rm aff}(H).
\end{equation}
Then it is easy to check that $T_H$ is an affine isometry with $T_H(H)\subset (B)_\mathcal N.$

Case II. $A$ is contained in a finite dimensional subspace of $X$, but its cardinality $A^\sharp=\infty$. Let ${\rm dim}[{\rm aff}(A)]=n$, and $(z_0,z_1,\cdots,z_n)\subset A$ be a maximal affinely independent subset of $A$. Then ${\rm aff}(z_0,z_1,\cdots,z_n)={\rm aff}(A)$.
Fix any sequence $(w_1,w_2,\cdots)\subset A$, which is dense in $A$.  For each $j\in\N$, let $$H_j={\rm co}(z_0,z_1,\cdots,z_n,w_1,\cdots,w_j).$$ Since $H_j$ contains the $n$-simplex ${\rm co}(z_0,z_1,\cdots,z_n)$,  $H_j$ is an $(n,m)$-convex polyhedron for some $m\geq0$ and satisfying $H_j\subset H_{j+1}$.
Let $${\rm vert}(H_j)=(x_0,x_1,\cdots,x_n,u_1,\cdots,u_m) (\subset (z_0,z_1,\cdots,z_n,w_1,\cdots,w_j))$$
such that $S\equiv{\rm co}(x_0,x_1,\cdots,x_n)$ is an $n$-simplex.

Since $A$ is strongly finitely representable in $B$, by Lemma 2.6, there exist an $(n,m)$-polyhedron $$H^\prime_j=H^\prime_j(y_{0j},y_{1j},\cdots,y_{nj},v_{1j},\cdots,v_{mj})\subset B$$
satisfying that  $S_j^\prime={\rm co}(y_{0j},y_{1j},\cdots,y_{nj})$ is an $n$-simplex,
 and an affine mapping $$T_{A,j}: {\rm aff}(A)\rightarrow {\rm aff}(B)$$ such that
\begin{equation}
(1-\frac{1}j)\|x-y\|\leq\|T_{A,j}x-T_{A,j}y\|\leq(1+\frac{1}j)\|x-y\|,\;\forall x, y\in {\rm aff}(A),
\end{equation}
with
\begin{equation}
T_{A,j}(S)=S^\prime_j,\;\;{\rm and\;\;} d_\mathfrak{H}\Big({\rm vert}(T_{A,j}(H_j)),{\rm vert}(H_j^\prime)\Big)<\frac{1}j.
\end{equation}

For any fixed free ultrafilter $\mathcal N$ on $\N$, we define an affine mapping $T_A: {\rm aff}(A)\rightarrow ({\rm aff}(B))_\mathcal N$ by
\begin{equation}
T_A(x)=(T_{A,j}(x))_\mathcal N,\;\;x\in{\rm aff}(A).
\end{equation}
Then it follows from (3.4) and (3.5) that $T_A$ is an affine isometry with $T_A(x)\in (B)_\mathcal N$ for all $x\in (w_j)_{j=1}^\infty$. Since $(B)_\mathcal N$ is always closed in $(X)_\mathcal N$ and since $(w_j)_{j=1}^\infty$ is dense in $A$, we obtain that $T_A(A)\subset(B)_\mathcal N.$

Case III. ${\rm span}(A)$ is an infinite dimensional subspace. Let $$J=\{F (\neq\emptyset)\subset A: {\rm dim}[{\rm span}(F)]<\infty\}.$$
Let $\mathcal F_J$ be the family
 consisting of all cofinal subsets of $J$ ordered by set inclusion. Then it is a free filter on $J$. Choose any  free  ultrafilter  $\mathcal V$ on $J$ containing $\mathcal F_J$, and fix a free ultrafilter $\mathcal N$ on $\N$.
By Cases I and II that we have just proven, for each $F\in J$, there is an affine isometry $T_F: {\rm aff}(F)\rightarrow[{\rm aff}(B)]_\mathcal N$ such that
\begin{equation}
T_F(F)=(T_{F,j}(F))_\mathcal N\subset(B)_\mathcal N.
\end{equation}

Next, fix any $y_0\in (B)_\mathcal N$, and  let $\tilde{T}_F: {\rm aff}(A)\rightarrow ({\rm aff}(B))_\mathcal N$ be defined by
\begin{equation}
\tilde{T}_F(x)= T_F(x),\;{\rm if\;}x\in{\rm aff}(F);\;=y_0,\;{\rm otherwise}.
\end{equation}

Finally, let $\mathcal U$ be the product (free) ultrafilter $\mathcal N\times\mathcal V$, and let $T: {\rm aff}(A)\rightarrow ({\rm aff}(B))_\mathcal U$ be defined for $x\in{\rm aff}(A)$ by
 \begin{equation}
 T(x)=(\tilde{T}_F(x))_\mathcal V=\big((T_{F,j}(x))_\mathcal N\big)_\mathcal V=\big(T_{F,j}(x)\big)_\mathcal U.
 \end{equation}
 Then $T$ is an affine isometry desired.
\end{proof}

\section{Strongly finite representability of sets in their countable subsets }
In this section, we will show that every subset of a Banach space is strongly finitely representable in a countable subset of it. For the proof of this result, we will need a sequence of lemmas.

For $m$ elements $a_j: j=1,2,\cdots,m$ of a set $A$, $(a_1,a_2,\cdots,a_m)$ is either to denote the subset $\{a_j: j=1,2,\cdots,m\}$, or, the ``vector" $(a_1,a_2,\cdots,a_m)$ in the Cartesian product $A^m$. We often blur the distinction if it arises no confusion.

Given $1\leq p\leq \infty$, let $\|\cdot\|_p$ be the $\ell_p$-norm defined on $\mathbb R^n$, i.e.
\begin{equation}
\|x\|_p=(\sum_{j=1}^n|x(j)|^p)^{\frac{1}p},\;\;x=(x(1),x(2),\cdots, x(n))\in\mathbb R^n.
\end{equation}
Let $K$ be the closed unit ball of the  space $\ell^n_\infty\equiv(\R^n,\|\cdot\|_\infty)$,  $$\Omega=\big\{\||\cdot\||\; {\rm is\;a \;seminorm \;on\;}\R^n\;{\rm with\;} \||\cdot\||\leq\|\cdot\|_1\big\},$$
and $C(K)$ be the Banach space of all continuous functions on $K$ with the sup-norm $\|\cdot\|_\infty$.
\begin{lemma}
Assume that $K$ and $\Omega$ are defined as above. Then $\Omega$ is a convex compact set of $C(K)$.
\end{lemma}
\begin{proof}
Convexity of $\Omega$ is trivial. Note that for every $\||\cdot\||\in\Omega$,  $$\||x\||\leq\|x\|_1\leq n\|x\|_\infty,\;\;\forall\;x\in\R^n.$$ This says that the Lipschitz norm of $\||\cdot\||$ is bounded by $n$.
Thus, $\Omega$ is relatively compact in $C(K)$. Since the uniform limit of a sequence in $\Omega$ is again a seminorm dominated by $\|\cdot\|_1$, this entails that $\Omega$ is closed.
\end{proof}
\begin{lemma}
For every $n$ dimensional (real) normed space $X$, there exists a norm $\||\cdot\||$ on $\R^n$ satisfying
\begin{equation}
\|\cdot\|_\infty\leq \||\cdot|\|\leq\|\cdot\|_1,
\end{equation}
and there is a linear isometry $T: X\rightarrow (\R^n,\||\cdot|\|).$
\end{lemma}
\begin{proof}
Since $\dim (X)=n$, by the Auerbach theorem (see, for instance, \cite[Prop.1.c.3]{lin}) there exist $(x_j)_{j=1}^n\subset S_X$, $(x^*_i)_{i=1}^n\subset S_{X^*}$ such that
\begin{equation}
\langle x^*_i,x_j\rangle=\delta_{ij}\; (=1,\;{\rm if\;}i=j;\;=0,\;{\rm if}\;i\neq j).
\end{equation}
Note that $(x_j)_{j=1}^n$ is a basis of $X$. For every $x\in X$, there is a unique $a=(a_j)_{j=1}^n\in\R^n$ such that $x=\sum_{j=1}^na_jx_j$. This and (4.3) imply
$|a_j|\leq  1,\;j=1,2,\cdots,n.$ Let
\begin{equation}\||a\||=\|x\|,\;\;{\rm and}\;\;T(x)=a,\; \;{\rm for\;} x=\sum_{j=1}^na_jx_j\in X.\end{equation}
Then $T: X\rightarrow(\R^n,\||\cdot\||)$ is a linear isometry.

Given $a=(a_j)_{j=1}^n\in\R^n$, and $x=\sum_{j=1}^na_jx_j$,  since $$\|a\|_1=\sum_{j=1}^n|a_j|\|x_j\|\geq\|x\|=\|\sum_{j=1}^na_jx_j\|\geq\max_{1\leq i\leq n}\{\langle \pm x^*_i, \sum_{j=1}^na_jx_j\rangle\}=\|a\|_\infty,$$
(4.2) follows.
\end{proof}
\begin{lemma}
Let $p_1, p_2\in C(K)$ be two norms on $\R^n$ with $p_1\geq\|\cdot\|_\infty$, and let $\eps>0$. If $\|p_1-p_2\|_{C(K)}\leq\eps$, i.e.
\begin{equation}
\big|p_1(x)-p_2(x)\big|\leq\eps,\;\;\forall \;x\in K,\end{equation}
then \begin{equation}(1-\eps)p_1(x)\leq p_2(x)\leq(1+\eps)p_1(x),\;\forall \;x\in\R^n.\end{equation}
\end{lemma}

\begin{proof}
It follows from (4.5) that
\begin{equation}p_1(x)-\eps\leq p_2(x)\leq p_1(x)+\eps,\;\;\forall \;x\in K.\end{equation}
Note that $S({p_1})\equiv\{x\in\R^n: p_1(x)=1\}\subset K$. Then it follows from (4.7) that
 \begin{equation}(1-\eps)p_1(x)\leq p_2(x)\leq(1+\eps)p_1(x),\;\forall \;x\in S(p_1).\end{equation}
 Consequently, (4.6) holds.
\end{proof}
%Recall that the Hausdorff metric $d_\mathfrak{H}$ on the family  $\mathscr B(X)$ of all nonempty bounded subsets of a Banach space $X$ is defined for $A, B\in \mathscr B(X)$ by
%\begin{equation}
%d_\mathfrak{H}(A, B)=\max\{\sup_{a\in A}\inf_{b\in B}\|a-b\|,\;\sup_{b\in B}\inf_{a\in A}\|a-b\|\},
%\end{equation}
  Recall that every $n$-dimensional convex polyhedron (or, $n$-DCP, for short) $H$ of a Banach space $X$ is an $(n,m)$-polyhedron for some integer $m\geq0$, i.e. $H$ has the following form: $H={\rm co}(x_0,x_1,\cdots,x_n, u_1,\cdots,u_m)$,  with $${\rm vert}(H)=(x_0,x_1,\cdots,x_n, u_1,\cdots,u_m),$$ and satisfying that $(x_0,x_1,\cdots,x_n)$ is a maximal affinely independent subset of ${\rm vert}(H)$. %In this case, we call also  the $n$-dimensional convex polyhedron $H$ an $(n,m)$-polyhedron.

 Given a nonempty subset $A$ of $X$, and for each pair of integers $m,n\geq0$, we denote by
\begin{equation}
\mathcal H_{n,m}(A)=\{H\subset X \;{is\;an\;}(n,m)\text{-polyhedron and with } {\rm vert}(H)\subset A\},
\end{equation}
\begin{equation}
\mathcal H_{n}(A)=\{H\subset X \;{is\;an\;}n\text{-DCP and with } {\rm vert}(H)\subset A\},
\end{equation}
and by
\begin{equation}
\mathcal H(A)=\{H\subset X \;{is\;a\;}\text{convex polyhedron and with } {\rm vert}(H)\subset A\}.
\end{equation}
%If, in addition, $0\in A$, then we use $\mathcal H_{n,m,0}(A)$ to denote the set of all $(n,m)$-polyhedrons $H={\rm co}(0,x_1,\cdots,x_n,u_1,\cdots,u_m)$ contained in  $\mathcal H_{n,m}(A)$, i.e. all those $(n,m)$-polyhedrons whose  vertex sets contain a maximal affinely independent subset of the form $(x_0,x_1,\cdots,x_n)$.
%Then we obtain
\begin{lemma}
With notations as above, for every subset  $A$ of a Banach space $X$, we have
\begin{equation}
\mathcal H_{n}(A)=\bigcup_{m=0}^\infty\mathcal H_{n,m}(A),\;\;{\rm and\;\;} \mathcal H(A)=\bigcup_{n=0}^\infty\mathcal H_{n}(A).
\end{equation}
\end{lemma}
Given two nonnegative integers $m,n$ and $r>0$, let $\mathcal H_{n,m}\equiv\mathcal H_{n,m}(rB_{\ell^n_\infty}),$ i.e. the set of all $(n,m)$-polyhedrons $H$ contained in $\ell^n_\infty\equiv(\R^n,\|\cdot\|_\infty)$ of the form $$H={\rm co}(x_0,x_1,\cdots,x_n, u_1,\cdots,u_m)$$ with its vertexes $${\rm vert}(H)=(x_0,x_1,\cdots,x_n, u_1,\cdots,u_m)\subset rB_{\ell^n_\infty}\equiv\{rx: x\in B_{\ell^n_\infty}\}$$ such that $(x_0,x_1,\cdots,x_n)$ is a maximal affinely independent subset of ${\rm vert}(H)$. %=(x_0,x_1,\cdots,x_n, u_1,\cdots,u_m)$.
Since $d_\mathfrak{H}$-convergence of a sequence $({\rm vert}(H_j))$ for $(H_j)\subset\mathcal  H_{n,m}$ is equivalent to the convergence of the corresponding vertex vector sequence $$\{(x_{0,j},x_{1,j}, \cdots, x_{n,j},u_{1,j},\cdots,u_{m,j})\}\;{\rm in}\; \ell_\infty^{n(n+1+m)}$$ within some permutations (Lemma 2.4),  the following result follows easily.
\begin{lemma} For each pair $m,n$ of nonnegative integers,
 $({\rm vert}(\mathcal H_{n,m}), d_\mathfrak{H})$ is relatively compact in $(\mathscr B(\ell_\infty^n), d_\mathfrak{H})$, where ${\rm vert}(\mathcal H_{n,m})=\{{\rm vert}(H): H\in\mathcal H_{n,m}\}$.
\end{lemma}
 Assume that $A$ is a nonempty bounded subset of a Banach space $X$, and that $n,m$ are two nonnegative integers. For each $$H={\rm co}(x_0,x_1,\cdots,x_n,u_1,\cdots,u_m)\in\mathcal H_{n,m}(A),$$ let $X_H={\rm span}(H)$. Then $X_H={\rm span}\{x_0,x_1,\cdots,x_n\}$ is either $n$, or, $n+1$ dimensional. Next, put
 $$\mathcal H_{n,m,0}(A)=\{H\in\mathcal H_{n,m}(A): \;{\rm dim}(X_H)=n\}$$ and
 $$\mathcal H_{n,m,1}(A)=\{H\in\mathcal H_{n,m}(A): \;{\rm dim}(X_H)=n+1\}.$$
 For each $H\in\mathcal H_{n,m,0}(A)$ (resp. $\mathcal H_{n,m,1}(A)$), let the linear isometry $T: X_H\rightarrow (\R^n,\||\cdot\||)$ (resp.  $X_H\rightarrow (\R^{n+1},\||\cdot\||)$) and the norm $\||\cdot\||$ with respect to $X_H$ be defined as Lemma 4.2. We write the linear isometry $T$ (resp. the norm $\||\cdot\||$) as $T_H$ (resp. $|\cdot|_H$) for distinction. Keep these notations in mind. Then we have the following result.
\begin{lemma}
Suppose that $A$ is a bounded subset of a Banach space $X$. Let  $r=\sup_{a\in A}\|a\|$. Then for any fixed $n,m\in\N$,
$$\mathcal T_{n,m,0}(A)\equiv\{{\rm vert}(T_H(H))\subset\ell^n_\infty: H\in\mathcal H_{n,m,0}(A)\}$$
$$({\rm resp.\;}\;\mathcal T_{n,m,1}(A)\equiv\{{\rm vert}(T_H(H))\subset\ell^{n+1}_\infty: H\in\mathcal H_{n,m,1}(A)\})$$
 is a bounded subset of $(\mathscr{B}(\ell^n_\infty), d_\mathfrak{H})$ (resp. $(\mathscr{B}(\ell^{n+1}_\infty), d_\mathfrak{H})$) and bounded by $2r$. Consequently, by Lemma 4.5, $\mathcal T_{n,m}(A)=\mathcal T_{n,m,0}(A)\cup \mathcal T_{n,m,1}(A)$ is relatively compact.
\end{lemma}
\begin{proof}
Let $r=\sup_{a\in A}\|a\|$. Then $A\subset rB_X$. For each fixed $H\in\mathcal H_{n,m,0}(A)$ (resp. $H\in\mathcal H_{n,m,1}(A)$), since $T_H: X_H\rightarrow (\R^n,|\cdot|_H)$ (resp. $X_H\rightarrow (\R^{n+1},|\cdot|_H)$) is a linear isometry,
$$\sup_{a\in T_H(H)}\|a\|_\infty\leq\sup_{a\in T_H(H)}|a|_H=\sup_{a\in T_H(H)}\|T^{-1}_H(a)\|=\sup_{h\in H }\|h\|\leq r.$$
Therefore, $$\bigcup\big\{T_H(H): H\in\mathcal H_{n,m,0}(A)\big\}\subset\ell^n_\infty$$
$$({\rm resp.\;}\bigcup\big\{T_H(H): H\in\mathcal H_{n,m,1}(A)\big\}\subset\ell^{n+1}_\infty)$$ is bounded by $r$. Consequently, $$\mathcal T_{n,m,0}(A)=\{{\rm vert}(T_H(H)): H\in\mathcal H_{n,m,0}(A)\}\subset \mathscr{B}({\ell_\infty^{n}})$$
$$({\rm resp. \;}\mathcal T_{n,m,1}(A)=\{{\rm vert}(T_H(H)): H\in\mathcal H_{n,m,1}(A)\}\subset\mathscr{B}({\ell_\infty^{n+1}}))$$
is bounded by $2r$.
\end{proof}

 Now, we use $K_n$ to denote the closed unit ball of the space $\ell_\infty^n$,  $\Omega_n\subset C(K_n)$ is the set of all seminorms on $\R^n$ dominated by $\|\cdot\|_1$ (defined as Lemma 4.1), and for a subset $A\subset X$, let $\mathcal T_{n,m,0}(A)$ and  $\mathcal T_{n,m,1}(A)$ be defined as Lemma 4.6.
 \begin{lemma}
 Suppose that $A$ is a bounded subset of a Banach space $X$. Then  the products $\mathcal P_{n,m,0}(A)\equiv\mathcal T_{n,m,0}(A)\times\Omega_n$ and $\mathcal P_{n,m,1}(A)\equiv\mathcal T_{n,m,1}(A)\times\Omega_{n+1}$ are again relatively compact.
 \end{lemma}
 \begin{proof}
   By Lemmas 4.1 and 4.6, $\Omega_n$, $\Omega_{n+1}$, $\mathcal T_{n,m,0}(A)$ and  $\mathcal T_{n,m,1}(A)$  are relatively compact. Therefore, the products $\mathcal P_{n,m,0}(A)\equiv\mathcal T_{n,m,0}(A)\times\Omega_n$ and $\mathcal P_{n,m,1}(A)\equiv\mathcal T_{n,m,1}(A)\times\Omega_{n+1}$ are again relatively compact.
 \end{proof}
\begin{lemma}
Assume that $A$ be a bounded subset of a Banach space $X$.
Given $U_0=({\rm vert}(T_0H_0),|\cdot|_0), U_
j=({\rm vert}(T_jH_j),|\cdot|_j)\in\mathcal P_{n,m,0}(A)$ (resp. $\mathcal P_{n,m,1}(A)$ ) $\;j=1,2,\cdots,$ if $U_j\rightarrow U_0$ in $\mathcal P_{n,m,0}(A)$ (resp. $\mathcal P_{n,m,1}(A)$),
then ${\rm vert}(H_0)$ is strongly (finitely) representable in $\bigcup_{j=1}^\infty{\rm vert}(H_j)$.
\end{lemma}

\begin{proof}
%Let $d_{\frak{H},n}$ (resp. $d_{\frak{H},n}$) denote the Hausdorff metric of $\mathscr B(\ell_\infty^n)$
We first emphasize that  $U_j\rightarrow U_0$ in $\mathcal P_{n,m,0}(A)$ (resp. $\mathcal P_{n,m,1}(A)$) is equivalent to that the vertex sequence $(U_j)$ of the $(n,m)$-polyhedrons $T_jH_j$ satisfies
 \begin{equation}
 U_j\rightarrow U_0 \;{\rm in\;} (\mathscr{B}(\ell^n_\infty),d_\mathfrak H)\;\;\;({\rm resp.\;\;}(\mathscr{B}(\ell^{n+1}_\infty),d_\mathfrak H)),
 \end{equation}
 and the sequence $(|\cdot|_j)$ of the norms $|\cdot|_j$ on $\R^n$ (resp. $\R^{n+1}$) satisfies
 \begin{equation}
 |\cdot|_j\rightarrow |\cdot|_0\;{\rm in\;} C(K_n)\;({\rm resp.\;} C(K_{n+1})).
 \end{equation}

Let $$H_j={\rm co}(h_{j,0}, h_{j,1},\cdots,h_{j,n},h_{j,n+1},\cdots,h_{j,n+m}),\;j=0,1,\cdots,$$ be $(n,m)$-polyhedrons satisfying that $${\rm vert}(H_j)=(h_{j,0}, h_{j,1},\cdots,h_{j,n},h_{j,n+1},\cdots,h_{j,n+m})\subset A$$ and that
$(h_{j,0}, h_{j,1},\cdots,h_{j,n})$ is a maximal affinely independent subset of ${\rm vert}(H_j)$. We denote by $$S(A,j)={\rm co}(h_{j,0}, h_{j,1},\cdots,h_{j,n}), j=0,1,\cdots,$$ the $n$-simplexes in $X$ with ${\rm vert}(S(A,j))=(h_{j,0}, h_{j,1},\cdots,h_{j,n})\subset A$.

Note that each $W_j\equiv T_jH_j$ is again an $(n,m)$-polyhedron of the form $$W_j={\rm co}(a_{j,0}, a_{j,1},\cdots,a_{j,n},a_{j,n+1},\cdots,a_{j,n+m}),\; j=0,1,\cdots,$$
satisfying that for each $j$, $(a_{j,0}, a_{j,1},\cdots,a_{j,n})$ is a maximal affinely independent subset of the whole vertexes $(a_{j,0}, a_{j,1},\cdots,a_{j,n},a_{j,n+1},\cdots,a_{j,n+m})=U_j$.
Then  \begin{equation}{\rm vert}(W_j)=U_j\rightarrow U_0\equiv{\rm vert}(W_0)\end{equation} in the Hausdorff metric $d_\mathfrak{H}$ is equivalent to that there exists a sequence of permutations $(\pi_j)$ of the set $(0,1,\cdots, n, n+1,\cdots,n+m)$ such that $$(a_{j,\pi_j(0)},a_{j,\pi_j(1)},\cdots, a_{j,\pi_j(n)},a_{j,\pi_j(n+1)},\cdots,a_{j,\pi_j(n+m)})$$ converges to $$(a_{0,0}, a_{0,1},\cdots,a_{0,n},a_{0,n+1},\cdots,a_{0,n+m})$$ in the  space $\ell_\infty^{n(n+1+m)}$ (resp. $\ell_\infty^{(n+1)(n+1+m)}$). Without loss of generality, we can assume that \begin{equation}(a_{j,0},a_{j,1},\cdots,a_{j,n},a_{j,n+1},\cdots, a_{j,n+m})\rightarrow(a_{0,0}, a_{0,1},\cdots,a_{0,n},a_{0,n+1},\cdots,a_{0,n+m}),\end{equation} as $j\rightarrow\infty$. For each fixed $j\geq0$, let \begin{equation}S_j={\rm co}(a_{j,0},a_{j,1},\cdots,a_{j,n}).\end{equation} \\

Case I. $U_0=({\rm vert}(T_0H_0),|\cdot|_0), U_
j=({\rm vert}(T_jH_j),|\cdot|_j)\in\mathcal P_{n,m,0}(A)$. In this situation, we know that ${\rm aff}(S_j)={\rm span}(S_j)=\R^n.$
We can define an affine isomorphism $L_j: \ell_\infty^n\rightarrow\ell_\infty^n$ satisfying
\begin{equation}L_j(a_{j,k})=a_{0,k},\; k=0,1,\cdots,n.\end{equation}
It follows from (4.16)-(4.18)
\begin{equation} L_j(S_j)=S_0,\;{\rm and\;} L_j\rightarrow id_{\ell^n_\infty}+v_0\end{equation}
for some vector $v_0\in\ell^n_\infty$ with respect to the norm of $\ell^n_\infty$ (hence, with respect to the equivalent norm $|\cdot|_0$)   as $j\rightarrow\infty$.
Therefore,
\begin{equation}
d_\mathfrak{H}\Big(L_j(U_j), U_0\Big)\rightarrow 0,\; \;{\rm in}\; \Big(\mathscr{B}(\ell_\infty^n),d_\mathfrak{H}\Big),\;{\rm as\;}\;j\rightarrow\infty.
\end{equation}
 Consequently,
\begin{equation}
d_\mathfrak{H}\Big(L_j^{-1}(U_0),U_j\Big)\rightarrow 0,\;\;{\rm in}\; \Big(\mathscr{B}(\R^n,|\cdot|_j),d_\mathfrak{H}\Big),{\rm as\;}j\rightarrow\infty.
\end{equation}
Since $T_j: X_j\equiv{\rm span}(S(A,j))\rightarrow (\R^n,|\cdot|_j)$ are linear isometries, we obtain  further that
\begin{equation}\begin{array}{cc}
              d_\mathfrak{H}\Big(L_j^{-1}(U_0),U_j\Big)=d_\mathfrak{H}\Big(T_j^{-1}L_j^{-1}U_0,T_j^{-1}U_j\Big)\;\;\;\;\;\;\;\;\;\;\;\;\;\;\;\;\\
                =d_\mathfrak{H}\Big((T_j^{-1}L_j^{-1})U_0,{\rm vert}(H_j)\Big)\rightarrow 0,\;{\rm as\;}j\rightarrow\infty.

                  \end{array}
\end{equation}
Lemmas 4.2, 4.3 and  $|\cdot|_j\rightarrow|\cdot|_0$ in $\Omega_n\subset C(K_n)$ and (4.19) imply that for every $\eps>0$, there exists $j_0\in\N$ such that for all $a,b\in\R^n$(whenever $j\geq j_0$)
\begin{equation}
(1-\eps)|a-b|_j\leq|L_ja-L_jb|_0\leq (1+\eps)|a-b|_j,
\end{equation}
and
\begin{equation}
(1-\eps)|L^{-1}_ja-L^{-1}_jb|_j\leq|a-b|_0\leq (1+\eps)|L^{-1}_ja-L^{-1}_jb|_j.
\end{equation}
Let $X_j={\rm span}(S(A,j))(={\rm aff}(H_j)\subset X)$.
Since $L_j(S_j)=S_0\subset T_0H_0$, and since $$T_j: X_j\rightarrow (\R^n,|\cdot|_j),\;j=0,1,\cdots,$$ are  linear isometries, it follows from (4.24) that for all $j\geq j_0$
\begin{eqnarray} &(1-\eps)\|T_j^{-1}L^{-1}_ja-T_j^{-1}L^{-1}_jb\|\leq\|T_0^{-1}a-T_0^{-1}b\|\;\;\;\;\;\;\;\;\;\;\\\nonumber
&\leq (1+\eps)\|T_j^{-1}L^{-1}_ja-T_j^{-1}L^{-1}_jb\|.\;\;\;\;\;\\\nonumber
\end{eqnarray}
%$$(1-\eps)\|T_m^{-1}L^{-1}_ma-T_m^{-1}L^{-1}_mb\|\leq\|T_0^{-1}a-T_0^{-1}b\|$$
%$$\leq (1+\eps)\|T_m^{-1}L^{-1}_ma-T_m^{-1}L^{-1}_mb\|.$$
Let $V_j=T_j^{-1}L^{-1}_jT_0$.  Then $$V_j: X_0\equiv{\rm span}(S(A,0))\rightarrow X_j\equiv{\rm span}(S(A,j))$$ is an isomorphism, which satisfies \begin{equation}V_j(S(A,0))=S(A,j).\end{equation}
We write $T_0^{-1}a=x, T_0^{-1}b=y.$ It follows from (4.25) that for all $x,y\in{\rm span}(S(A,0))={\rm aff}(H_0)$,
\begin{equation}(1-\eps)\|V_jx-V_jy\|\leq\|x-y\|\leq(1+\eps)\|V_jx-V_jy\|.\end{equation}
On the other hand, note that ${\rm vert}(T_0H_0)=U_0$, and that $${\rm vert}(V_j(H_0))={\rm vert}(T_j^{-1}L^{-1}_jT_0(H_0))=T_j^{-1}L^{-1}_j(U_0).$$ It follows from (4.22) that
\begin{equation}
d_\mathfrak{H}\Big({\rm vert}(V_jH_0),{\rm vert}(H_j)\Big)\rightarrow 0,\;{\rm as\;}j\rightarrow\infty.
\end{equation}

Thus, it follows from (4.26), (4.27) and (4.28) that ${\rm vert}(H_0)$ is strongly (finitely) representable in $\bigcup_{j=1}^\infty{\rm vert}(H_j)$.\\

Case II.  $U_0=({\rm vert}(T_0H_0),|\cdot|_0), U_
j=({\rm vert}(T_jH_j),|\cdot|_j)\in\mathcal P_{n,m,1}(A)$. In this case, for each integer $j\geq0$, $X_j\equiv{\rm span}(H_j)$ is an $n+1$ dimensional subspace, and $A_j\equiv{\rm aff}(H_j)$ is an $n$ dimensional (proper) affine subspace of $X_j$, and $T_j: X_j\rightarrow (\R^{n+1},|\cdot|_j)$ is a linear isometry.
We use $T^r_j$ to denote $T_j|_{A_j}$, the restriction of $T_j$ restricted to $A_j$, and let  $R_j=T_j(A_j)\subset (\R^{n+1},|\cdot|_j)$, the range of $A_j$ under $T_j$. Therefore, $T^r_j: A_j\rightarrow R_j$ is a surjective affine isometry.

It follows from (4.13), (4.14), (4.16) and (4.17) that \begin{equation}a_{j,k}\rightarrow a_{0,k},\;{\rm  in\;} \ell^{n+1}_\infty\;{\rm for} \;k=0,1,\cdots,n+m,\;{\rm as \;} j\rightarrow\infty,\end{equation}
\begin{equation}
S_j={\rm co}(a_{j,0},a_{j,1},\cdots,a_{j,n})\;(j=0,1,\cdots)\; {\rm are\;}n\text{-}{\rm simplexes \;in\;}\ell_\infty^{n+1}.
\end{equation}
\begin{equation}
{\rm vert}(T_jH_j)=U_j=(a_{j,0},a_{j,1},\cdots,a_{j,n},\cdots, a_{j,n+m}),\;j=0,1,\cdots.
\end{equation}
\begin{equation}
|\cdot|_j\rightarrow |\cdot|_0,\;{\rm in\;}C(K_{n+1})\;{\rm as\;}j\rightarrow\infty,
\end{equation}
and
\begin{equation}
d_\mathfrak H\Big(U_j,U_0\Big)\rightarrow 0,\;{\rm in\;}\mathscr{B}(\ell_\infty^{n+1}),\;{\;as\;}j\rightarrow\infty.
\end{equation}

For each $j\in\N$, we define an affine isomorphism $L_j: R_j\rightarrow R_0$ for $\sum_{k=0}^n\alpha_ka_{j,k}\in R_j$ (where $\alpha_k\in\R, \;k=0,1,\cdots,n,$ with $\sum_{k=0}^n\alpha_k=1$) by
\begin{equation}
L_ja_{j,k}=a_{0,k}, \;\;L_j(\sum_{k=0}^n\alpha_ka_{j,k})=\sum_{k=0}^n\alpha_kL_ja_{j,k},
\end{equation}
Therefore,
\begin{equation}
L_j(S_j)=S_0,\;j=,1,2,\cdots,
\end{equation}
where $S_j$ are defined as (4.17).
It follows from (4.29), (4.32) and (4.34) that
\begin{equation}
\|L_j\|,\;\|L_j^{-1}\|\rightarrow 1,\;\;{\rm as\;}j\rightarrow\infty.
\end{equation}
(4.29), (4.31), (4.33) and (4.36) entail that
\begin{equation}
L_ja_{j,k}\rightarrow a_{0,k},\;{\rm for\;}k=1,\cdots,n,n+1\cdots,n+m,\;{\rm as\;}j\rightarrow\infty.
\end{equation}
Consequently,
\begin{equation}
d_\mathfrak{H}\Big(L_j(U_j),U_0\Big)\rightarrow0,\;{\rm in\;}\mathscr{B}(\R^{n+1},|\cdot|_0),\;{\rm as\;}j\rightarrow\infty.
\end{equation}
Equivalently,
\begin{equation}
d_\mathfrak{H}\Big({\rm vert}((L_jT_j)(H_j)),{\rm vert}(T_0(H_0))\Big)\rightarrow0,\;\;{\rm as\;}j\rightarrow\infty.
\end{equation}
Since $T_0: {\rm aff}(H_0)\rightarrow{\rm aff}(U_0)={\rm aff}(S_0)$ is an affine isometry, it follows from (4.39) that
\begin{equation}
d_\mathfrak{H}\Big({\rm vert}((T_0^{-1}L_jT_j)(H_j)),{\rm vert}(H_0)\Big)\rightarrow0,\;{\rm in\;}\mathscr{B}(X),\;{\rm as\;}j\rightarrow\infty.
\end{equation}
and this is equivalent to that
\begin{equation}
d_\mathfrak{H}\Big({\rm vert}(H_j),{\rm vert}((T_j^{-1}L_j^{-1}T_0)(H_0)\Big)\rightarrow0,\;{\rm in\;}\mathscr{B}(X),\;{\rm as\;}j\rightarrow\infty.
\end{equation}
Let\begin{equation}P_j=T_j^{-1}L_j^{-1}T_0, j=1,2,\cdots.\end{equation} Then $P_j: H_0\rightarrow H_j$ are affine isomorphisms and satisfy
\begin{equation}
\|P_j\|,\;\|P^{-1}\|\rightarrow 1,\;{\rm as\;}j\rightarrow\infty.
\end{equation}
Therefore,  it follows from (4.41), (4.42) and (4.43) that ${\rm vert}(H_0)$ is strongly (finitely) representable in $\bigcup_{j=1}^\infty{\rm vert}(H_j)$.
%This and (4.37) imply that
%\begin{equation}
%d_\mathfrak{H}(U_j,L_j^{-1}U_0)\rightarrow0,\;{\rm in\;}(\mathscr{B}(\R^{n+1},|\cdot|_j),d_\mathfrak{H}),\;{\rm as\;}j\rightarrow\infty.
%\end{equation}
\end{proof}

Now, we are ready to prove the main result of this section.

\begin{theorem}
Suppose that $A$ is a  subset of a Banach space $X$. Then there exists a countable subset $A_0$ such that $A$ is strongly finitely representable in $A_0$.
\end{theorem}
\begin{proof} Without loss of generality, we can assume that $A$ is bounded. Otherwise, let $A_k=A\cap(kB_X),\;k\in\N$. Then $A=\cup_kA_k$, and we consider the bounded sets $A_k$. If we show that for each $k$ there exists a countable subset $A_{k0}\subset A_k$ such that $A_k$ is strongly finitely representable in $A_{k0}$. Then we are done by taking $A_0=\cup_{k=1}^\infty A_{k0}$.

Assume again that $K_n$ is the closed unit ball of the space $\ell_\infty^n$,  $\Omega_n\subset C(K_n)$ is the set of all seminorms on $\R^n$ dominated by $\|\cdot\|_1$ (defined as Lemma 4.1); $\mathcal H_{n,m}(A)=\mathcal H_{n,m,0}(A)\cup\mathcal H_{n,m,1}(A)$ is defined as Lemma 4.6, i.e. the family of all $(n,m)$-polyhedrons $H$ with ${\rm vert}(H)\subset A$;  that $\mathcal T_{n,m,0}(A)$ and  $\mathcal T_{n,m,1}(A)$ are defined as Lemma 4.6;
   and that the products $\mathcal P_{n,m,0}(A)\equiv\mathcal T_{n,m,0}(A)\times\Omega_n$ and $\mathcal P_{n,m,1}(A)\equiv\mathcal T_{n,m,1}(A)\times\Omega_{n+1}$ are defined as Lemma 4.7.

  It suffices to show that for each fixed pair $n,m\in\N$, there exist two countable subsets $A_{n,m,0}$ and $A_{n,m,1}$ of $A$ such that every $(n,m)$-polyhedron $H$  in $\mathcal H_{n,m}(A)$ is strongly (finitely) representable in $A_{n,m}=A_{n,m,0}\cup A_{n,m,1}$.  It follows from Lemma 4.7 that for any fixed pair $n,m$ of nonnegative integers, both $\mathcal P_{n,m,0}(A)$ and $\mathcal P_{n,m,1}(A)$ are relatively compact. Consequently, there exist two sequences $(U_{n,m,0,j})_{j=1}^\infty$ and $(U_{n,m,1,j})_{j=1}^\infty$ with
$$U_{n,m,0,j}=({\rm vert}(T_{n,m,0,j}H_{0,j}),|\cdot|_{0,j})\in\mathcal P_{n,m,0}(A)$$ and $$U_{n,m,1,j}=({\rm vert}(T_{n,m,1,j}H_{1,j}),|\cdot|_{1,j})\in\mathcal P_{n,m,1}(A)$$
for some $(H_{0,j})_{j=1}^\infty, (H_{1,j})_{j=1}^\infty\subset \mathcal H_{n,m}(A)$, such that $(U_{n,m,0,j})_{j=1}^\infty$ is dense in $\mathcal P_{n,m,0}(A)$ and $(U_{n,m,1,j})_{j=1}^\infty$ is dense in $\mathcal P_{n,m,1}(A)$, where $T_{n,m,0,j}:{\rm span}(H_{0,j})\rightarrow(\R^n,|\cdot|_{0,j})$ and $T_{n,m,1,j}:{\rm span}(H_{1,j})\rightarrow(\R^{n+1},|\cdot|_{0,j})$ are linear isometries. It follows from Lemma 4.8 that each $(n,m)$-polyhedron $H\in\mathcal H_{n,m,0}(A)$ (resp. $\mathcal H_{n,m,0}(A)$)
is strongly finitely representable in the countable subset $A_{n,m,0}\equiv\bigcup_{j=1}^\infty{\rm vert}(H_{0,j})$ (resp. $A_{n,m,1}\equiv\bigcup_{j=1}^\infty{\rm vert}(H_{1,j})$.

Now, we finish the proof by taking $A_0=(\bigcup_{n,m=0}^\infty A_{n,m,0})\bigcup(\bigcup_{n,m=0}^\infty A_{n,m,1})$.

\end{proof}

%Note that if we take $m=0$ in the proof of Theorem 4.9, then we can obtain the following consequences.

\begin{corollary}
Suppose that $X$ and $Y$ are Banach spaces, and that $A\subset X$ and $B\subset Y$ are two subsets.

i) If $A$ is  finitely representable in $B$, then there is a countable subset $B_0$ of $B$ such that $A$ is finitely representable in $B_0$.

ii) If $A$ is strongly finitely representable in $B$, then there is a countable subset $B_0$ of $B$ such that $A$ is strongly finitely representable in $B_0$.
\end{corollary}
\begin{proof}
Note that $B$ is always strongly finitely representable in itself. By Theorem 4.9, there is a countable subset $B_0$ of $B$ such that $B$ is strongly finitely representable in $B_0$. Thus, $A$  is  finitely (resp. strongly finitely) representable in $B_0$ if $A$ is  finitely (resp. strongly finitely) representable in $B$.
\end{proof}

\begin{corollary}
Suppose that $X$ and $Y$ are Banach spaces. If $X$ is finitely representable in $Y$, then there is a separable subspace $Y_0$ such that $X$ is finitely representable in $Y_0$. Consequently, every Banach space is finitely representable in a separable subspace of it.
\end{corollary}

\section{Countable determination of the Kuratowski measure}

In this section, we will show the main result of the paper: the Kuratowski measure $\alpha$ defined on any metric space $M$ satisfies that for every bounded set $B\subset M$, there is a countable subset $B_0$ such that
$\alpha(B_0)=\alpha(B)$.
%Kuratowski measure of noncompactness of a bounded subset $B$ of a Banach space $X$ is equal to the corresponding  measure of an ultrapowers of $B$.
\begin{lemma}
Let $\alpha_M$ be the Kuratowski measure defined on a metric space $M$. Then there exist a Banach space $X$ and an isometric mapping $T: M\rightarrow X$
so that  $\alpha|_{T(M)}$ (the restriction of the Kuratowski measure $\alpha$ on $X$ restricted to $T(M)$) coincides with $\alpha_M$, i.e.
$$\alpha(T(B))=\alpha_M(B),\;{\rm for\; all}\; B\in\mathscr B(M).$$
\end{lemma}

\begin{proof}
It suffices to note that every metric space is isometric to a subset of a Banach space (see, for instance, \cite[Lemma 1.1]{ben}), and note that the Kuratowski measure of a subset of a metric space is invariant under isometric mappings.
\end{proof}

As we have already known that every metric space is isometric to a subset of a Banach space, we can blur the distinction between ``a metric space" and `` a subset of a Banach space" in this sequel.

Assume that $M$ is a metric space and $\mathcal U$ is an ultrafilter. For distinction, we use $\alpha_M$ to denote the Kuratowski measure on $M$, and $\alpha_\mathcal U$ to denote the  Kuratowski measure on the ultrapower $(M)_\mathcal U$ of $M$.

Given an ultrafilter $\mathcal U$ on an index set $I$, and a subset $A$ of a metric space $M$, since the canonical embedding $J: A\rightarrow (M)_\mathcal U$ is isometric, and since the Kuratowski measure is monotone non-decreasing in the order of set inclusion, we obtain
 $\alpha_M(A)=\alpha_{\mathcal U}(J(A))\leq\alpha_{\mathcal U}((A)_\mathcal U)$. On the other hand, for any $r>\alpha_M(A)$,  there exist a finite set $F$, and $A_j\subset M$ with ${\rm diam}(A_j)\leq r$ and with $A\subset\cup_{j\in F} A_j$. By Proposition 3.2,
 we have  $(A)_\mathcal U\subset(\cup_{j\in F} A_j)_\mathcal U=\cup_{j\in F}(A_j)_\mathcal U.$  This entails $\alpha_{\mathcal U}((A)_\mathcal U)\leq\max_{j\in F}\{\alpha_{\mathcal U}((A_j)_\mathcal U)\}\leq r.$
 Thus, the following result (belonging to Kaczor, Stachura, Walczuk and Zola \cite[Corollary. 2.3]{ka})  follows.

% Ideas of ultrapowers and 子集间有限表示 and 局部同胚是泛函分析的核心手段和工具...We shall begin by recalling these concepts.
 %Let $X$ be  a Banach space and $\mathfrak{B}(X)$ the family of bounded subsets of $X$.
 %Recall now the construction of an ultrapower of the space $X$. We use the same notations and terminology as use in [].
 %Let...
 %We shall adopt the notation ...
 %It is well known that...
 %The Kuratowski measure of noncompactness in the spaces $(X, \|\cdot\|)$ and $ (X_{\mathcal{U}}, \|\cdot\|_{\mathcal{U}})$ will be denoted by $\alpha $ and $\alpha_{\mathcal{U}}$.

%Let Y,X be Banach spaces. We say that Y is finitely representable (in short f.r.) in X if for any $\varepsilon$ > 0 and any finite dimensional subspace E CY
%there is a subspace F C X which is (1 +$\varepsilon$)-isomorphic to E.

%这一部分的概念还没补清楚，先整理主要结果。

\begin{lemma}
For any subset  $B$ of a metric space $M$, and for any ultrafilter $\mathcal U$, we have  \begin{equation}\alpha_M(B)=\alpha_\mathcal U[(B)_{\mathcal{U}}].\end{equation}

\end{lemma}

\bigskip

Now, we are ready to prove the main theorem of this paper.

\begin{theorem}
Suppose that $X$ is a metric space. Then for every bounded subset $B\subset X$, there is a countable subset $B_0$ of $B$ such that $\alpha(B_0)=\alpha(B)$.

\end{theorem}

\begin{proof}

By Lemma 5.1, we can assume that $X$ is a Banach space. Given $B\in\mathscr B(X)$, by Theorem 4.9, there exists a countable subset $B_0$ of $B$ such that $B$ is strongly finitely representable in $B_0$. Thus, $B$ is strongly finitely representable in $C\equiv{\rm co}(B_0)$. Note $\alpha(C)=\alpha(B_0).$
Applying Lemma 3.3, we can obtain a free ultrafilter $\mathcal U$, and an affine isometry $T: {\rm aff}(B)\rightarrow [{\rm aff}(B_0)]_\mathcal U$ such that $T(B)\subset (C)_\mathcal U$. It follows from Lemmas 5.1 and  5.2 that
\begin{eqnarray}
\alpha(B)=\alpha_\mathcal U[T(B)]\leq\alpha_\mathcal U[(C)_\mathcal U]=\alpha(C)=\alpha(B_0).
\end{eqnarray}
On the other hand, non-deceasing monotonicity of the Kuratowski measure $\alpha$ in the order of set inclusion entails that $\alpha(B)\geq\alpha(B_0)$.
Therefore, $\alpha(B_0)=\alpha(B)$.
\end{proof}

\begin{definition}
A measure of noncompactness $\mu$ on a metric space $M$ is said to be an ultra-isometric-invariant provided for every metric space $N$, every isometry $T: M\rightarrow N$  and every free ultrafilter $\mathcal U$, we have
\begin{equation}
\mu(T(B))=\mu(B)=\mu_\mathcal U((B)_\mathcal U),\;\;\forall\;B\in\mathscr B(M).
\end{equation}
\end{definition}

Clearly, the Kuratowski measure of noncompactness $\alpha$ and the Hausdorff measure of noncompactness $\beta$ on Hilbert spaces are ultra-isometric-invariant.

We can show the following result analogously.

\begin{theorem}
Suppose that $M$ is a metric space, and that $\mu$ is an ultra-isometric-invariant regular measure of noncompactness on $M$. Then for each $B\in\mathscr B(M)$, there is a countable $B_0\subset B$ such that $\mu(B_0)=\mu(B).$
\end{theorem}

\;
%{\bf Acknowledgements.}  The authors want to thank the referee for his/her constructive comments and helpful suggestions. The notion of strongly finite representation was motivated by his/her comments, suggestions  and the counterexample in Remark 2.9 was by his/her.

\bibliographystyle{amsalpha}

\end{document}